\documentclass[11pt,twoside]{amsart}
\usepackage{latexsym,amssymb,amsmath}
\numberwithin{equation}{section}
\newtheorem{theorem}{Theorem}[section]
\newtheorem{lemma}[theorem]{Lemma}
\newtheorem{proposition}[theorem]{Proposition}
\newtheorem{corollary}[theorem]{Corollary}

\theoremstyle{definition}
\newtheorem{definition}[theorem]{Definition}

\newtheorem{remark}[theorem]{Remark}
\newtheorem{example}[theorem]{Example}

\begin{document}


\title{Direct products in projective Segre codes} 

\author{Azucena Tochimani}
\address{
Departamento de
Matem\'aticas\\
Centro de Investigaci\'on y de Estudios
Avanzados del
IPN\\
Apartado Postal
14--740 \\
07000 Mexico City, D.F.
}
\email{tochimani@math.cinvestav.mx}

\author{Maria Vaz Pinto}
\address{
Departamento de Matem\'atica\\
Instituto Superior T\'ecnico\\
Universidade T\'ecnica de Lisboa\\
Avenida Rovisco Pais, 1\\
1049-001 Lisboa, Portugal.
}
\email{vazpinto@math.ist.utl.pt}

\author{Rafael H. Villarreal}
\address{
Departamento de
Matem\'aticas\\
Centro de Investigaci\'on y de Estudios
Avanzados del
IPN\\
Apartado Postal
14--740 \\
07000 Mexico City, D.F.
}
\email{vila@math.cinvestav.mx}

\thanks{The first author was partially supported by CONACyT. The second
author is a member of the Center for Mathematical Analysis, Geometry,
and Dynamical Systems, Departamento de Matematica, Instituto Superior
Tecnico, 
1049-001 Lisboa, 
Portugal. The third author was partially supported by SNI}


\begin{abstract} 
Let $K=\mathbb{F}_q$ be a finite field. 
We introduce a family of projective Reed-Muller-type codes 
called {\it projective Segre codes}. Using commutative algebra and
linear algebra methods, we study their basic parameters and show that they
are direct products of projective Reed-Muller-type codes. 
As a consequence we recover some results on
projective Reed-Muller-type codes over 
the Segre variety and over projective tori. 
\end{abstract}

\maketitle

\section{Introduction}\label{section-intro}

Reed-Muller-type evaluation codes have been extensively studied 
using commutative algebra methods 
(e.g., Hilbert functions, resolutions, Gr\"obner bases); see
\cite{carvalho,geil,algcodes} and the references therein. 
In this paper we use these methods---together with linear algebra
techniques---to study projective Segre codes over finite fields. 

Let $K$ be an arbitrary field, let $a_1,a_2$ be two positive integers,
let $\mathbb{P}^{a_1-1}$, $\mathbb{P}^{a_2-1}$
be projective spaces over $K$, and let $K[\mathbf{x}]=
K[x_1,\ldots,x_{a_1}]$, $K[\mathbf{y}]=
K[y_1,\ldots,y_{a_2}]$, $K[\mathbf{t}]=K[t_{1,1},\ldots,t_{a_1,a_2}]$ be
polynomial rings with the standard grading. If $d\in\mathbb{N}$, let
$K[\mathbf{t}]_d$ denote the set of homogeneous polynomials of
total degree $d$ in $K[\mathbf{t}]$, together with the zero
polynomial. Thus $K[\mathbf{t}]_d$ is a $K$-linear space and 
$K[\mathbf{t}]=\oplus_{d=0}^\infty K[\mathbf{t}]_d$. In this grading
each $t_{i,j}$ is homogeneous of degree one. 

Given $\mathbb{X}_i\subset\mathbb{P}^{a_i-1}$, $i=1,2$, denote by
$I(\mathbb{X}_1)$ (resp. $I(\mathbb{X}_2)$) the {\it vanishing ideal\/} of
$\mathbb{X}_1$ (resp. $\mathbb{X}_2$) generated by the
homogeneous polynomials of $K[\mathbf{x}]$ (resp. $K[\mathbf{y}]$)
that vanish at all points of $\mathbb{X}_1$ 
(resp. $\mathbb{X}_2$). 
The {\it Segre embedding} is given by
\begin{eqnarray*}
\psi\colon\mathbb{P}^{a_1-1}\times\mathbb{P}^{a_2-1}
&\rightarrow&\mathbb{P}^{a_1a_2-1}\\
([(\alpha_1,\ldots,\alpha_{a_1})],[(\beta_1,\ldots,
\beta_{a_2})])&\rightarrow&[(\alpha_i\beta_j)],
\end{eqnarray*}
where $[(\alpha_i\beta_j)]:=
[(\alpha_1\beta_1,\alpha_1\beta_2,\ldots,\alpha_1\beta_{a_2},\ldots,
\alpha_{a_1}\beta_1,\alpha_{a_1}\beta_2,\ldots,\alpha_{a_1}\beta_{a_2})]$.
The map $\psi$ is well-defined and injective
\cite[p.~13]{hartshorne}. 
The image of $\mathbb{X}_1\times\mathbb{X}_2$ 
under the map $\psi$, denoted by $\mathbb{X}$, is called the {\it
Segre product\/} of $\mathbb{X}_1$ and $\mathbb{X}_2$. The vanishing ideal $I(\mathbb{X})$ of 
$\mathbb{X}$ is a graded ideal of $K[\mathbf{t}]$, where the 
$t_{i,j}$ variables are ordered as
$t_{1,1},\ldots,t_{1,a_2},\ldots,t_{a_1,1},\ldots,t_{a_1,a_2}$.
The Segre 
embedding is used in algebraic geometry, among other applications, to show that 
the product of projective varieties is again a projective variety, 
see \cite[Lecture 2]{harris}. If $\mathbb{X}_i=\mathbb{P}^{a_i-1}$
for $i=1,2$, the set $\mathbb{X}$ is a projective variety and is called a 
{\it Segre variety\/}  \cite[p.~25]{harris}. The Segre embedding is
used in coding theory, among other applications, to study the generalized Hamming weights of some
product codes; see \cite{schaathun-willems} and the references 
therein. 

The contents of this paper are as follows. Let $K=\mathbb{F}_q$ be a
finite field. In Section~\ref{prelim-section} we recall two results about the basic
parameters and the second generalized Hamming weight of direct product codes (see
Theorems~\ref{dec29-14} and \ref{wei-yang-th}). Then we introduce the family of projective
Reed-Muller-type codes, examine their basic parameters, and explain the relation between 
Hilbert functions and projective Reed-Muller-type codes (see
Proposition~\ref{jan4-15}). For an arbitrary field $K$ we show 
that $K[\mathbf{t}]/I(\mathbb{X})$ is the Segre product of 
$K[\mathbf{x}]/I(\mathbb{X}_1)$ and $K[\mathbf{y}]/I(\mathbb{X}_2)$
(see Definition~\ref{segre-product-def} and Theorem~\ref{dec4-14-1}).
The Segre product is a subalgebra of
$$(K[\mathbf{x}]/I(\mathbb{X}_1))\otimes_K(K[\mathbf{y}]/I(\mathbb{X}_2)),$$
the tensor product algebra. Segre products have been studied by
many authors; see  \cite{Eisen,GRT,kahle-segre} and the
references therein. We give full proofs of two results for which we could not find a 
reference with the corresponding proof (see Lemma~\ref{jan12-15} and 
Theorem~\ref{dec4-14-1}). Apart from this all results of this section are
well known. 

If $K=\mathbb{F}_q$ is a finite field, we introduce a
family $\{C_\mathbb{X}(d)\}_{d\in\mathbb{N}}$ of projective
Reed-Muller-type 
codes that we call {\it projective Segre codes\/} (see
Definition~\ref{segre-code-def}). 
It turns out that $C_\mathbb{X}(d)$ is
isomorphic to $K[\mathbf{t}]_d/I(\mathbb{X})_d$, as $K$-vector spaces,
where $I(\mathbb{X})_d$ is equal to $I(\mathbb{X})\cap K[\mathbf{t}]_d$. 
Accordingly $C_{\mathbb{X}_1}(d)\simeq
K[\mathbf{x}]_d/I(\mathbb{X}_1)_d$ and $C_{\mathbb{X}_2}(d)\simeq
K[\mathbf{y}]_d/I(\mathbb{X}_2)_d$. In
Section~\ref{section-segre-codes} we study the basic parameters 
(length, dimension, minimum distance) and the second generalized Hamming
weight of projective Segre codes. Our main
result expresses the basic parameters of $C_{\mathbb{X}}(d)$ in terms of those of 
$C_{\mathbb{X}_1}(d)$ and $C_{\mathbb{X}_2}(d)$, 
and shows that $C_{\mathbb{X}}(d)$ is the direct
product of $C_{\mathbb{X}_1}(d)$ and $C_{\mathbb{X}_2}(d)$ (see
Theorem~\ref{azucena-maria-vila}); this means that the direct product of
two projective Reed-Muller-type codes of degree $d$ is again a
projective Reed-Muller-type code of degree $d$.  

Formulas for the basic parameters of affine and projective
Reed-Muller-type codes are known for a number of families
\cite{cicero-victor-hiram,delsarte-goethals-macwilliams,dias-neves,
duursma-renteria-tapia,geil-thomsen,gold-little-schenck,
GR,GRS,GRT,cartesian-codes,ci-codes,sorensen}. Since affine Reed-Muller-type codes can  
be regarded as projective Reed-Muller-type codes \cite{affine-codes}, our results can 
be applied to obtain explicit formulas for the basic parameters of
$C_{\mathbb{X}}(d)$ if $C_{\mathbb{X}_1}(d)$ is in one of these families and 
$C_{\mathbb{X}_2}(d)$ is in another of these families or both are in
the same family. 

As an application we recover some results on 
Reed-Muller-type codes over the  Segre variety and over 
 projective tori \cite{GR,G-SR-M,GRH,GRT}.  Indeed, if $\mathbb{X}_1=\mathbb{P}^{a_1-1}$ 
and $\mathbb{X}_2=\mathbb{P}^{a_2-1}$, using
Theorem~\ref{azucena-maria-vila} we recover the formula for the
minimum distance of $C_\mathbb{X}(d)$ given in 
\cite[Theorem~5.1]{GRT}. If $K^*=K\setminus\{0\}$
and $\mathbb{X}_i$ is the image of $(\mathbb{K}^*)^{a_i}$, under the 
map $(K^*)^{a_i}\rightarrow\mathbb{P}^{a_i-1}$, $x\rightarrow [x]$, we call
$\mathbb{X}_i$ a {\it projective torus\/} in $\mathbb{P}^{a_i-1}$. If $\mathbb{X}_i$ is a projective torus
for $i=1,2$, using Theorem~\ref{azucena-maria-vila} we recover the formula for the
minimum distance of $C_\mathbb{X}(d)$ given in \cite[Theorem~5.5]{GR}. 
In these two cases formulas for the
basic parameters of $C_{\mathbb{X}_i}(d)$, $i=1,2$,  
are given in \cite[Theorem~1]{sorensen} and
\cite[Theorem~3.5]{ci-codes}, respectively. We also recover the
formulas for the second generalized 
Hamming weight given in \cite[Theorem~5.1]{G-SR-M} and
\cite[Theorem~3]{GRH} (see Corollary~\ref{jan9-15}).

For all unexplained 
terminology and notation, and for additional information 
we refer to \cite{CLO,Sta1} (for the theory of Hilbert
functions) and to \cite{MacWilliams-Sloane,tsfasman} (for coding theory). Our main
references for commutative algebra and multilinear algebra are 
\cite{BHer,harris} and \cite[Appendix~2]{Eisen}, respectively. 

\section{Preliminaries}\label{prelim-section}
In this section, we 
present some of the results that will be needed throughout the paper
and introduce some more notation. We study direct product codes, and some of their
properties and characterizations. The families of
Reed-Muller-type codes and projective Segre codes are introduced here,
and their relation to tensor products and Hilbert functions is
discussed. 

\paragraph{\bf Generalized Hamming weights} 
Let $K=\mathbb{F}_q$ be a finite field and let $C$ be a $[s,k]$ {\it linear
code} of {\it length} $s$ and {\it dimension} $k$, 
that is, $C$ is a linear subspace of $K^s$ with $k=\dim_K(C)$.   

Given a subcode $D$ of $C$ (that is, $D$ is a linear subspace of $C$),
the {\it support\/} of $D$, denoted $\chi(D)$, is the set of non-zero positions of $D$, that is,  
$$
\chi(D):=\{i\,\vert\, \exists\, (a_1,\ldots,a_s)\in D,\, a_i\neq 0\}.
$$

The $r$th {\it generalized Hamming weight\/} of $C$, denoted
$\delta_r(C)$, is the size of the smallest support of an
$r$-dimensional subcode, that is,
$$
\delta_r(C):=\min\{|\chi(D)|\,\colon\, D\mbox{ is a linear subcode of
}C\mbox{ with }\dim_K(D)=r\}.
$$

Let $0\neq v\in C$. The {\it Hamming weight\/} of $v$, denoted by $\omega(v)$, is the number of non-zero
entries of $v$. If $\delta(C)$ is the {\it minimum distance\/} of $C$, that is, 
$\delta(C):=\min\{\omega(v)
\colon 0\neq v\in C)\}$, then 
note that $\delta_1(C)=\delta(C)$. The {\it weight hierarchy\/} of $C$ is the sequence
$(\delta_1(C),\ldots,\delta_k(C))$. According to \cite[Theorem~1]{wei}
the weight hierarchy is an
increasing sequence 
$$
1\leq\delta_1(C)<\delta_2(C)<\cdots<\delta_r(C)\leq s,
$$
and $\delta_r(C)\leq s-k+r$ for $r=1,\ldots,k$. For $r=1$ this is the
Singleton bound for the minimum distance. Generalized Hamming weights
have received a lot of attention; see
\cite{carvalho,geil,schaathun-willems,wei,wei-yang} and the references therein.

\paragraph{\bf Direct product codes and tensor products} 

Let $C_1\subset K^{s_1}$ and $C_2\subset K^{s_2}$ be two linear codes
over the finite field $K=\mathbb{F}_q$ and let $M_{s_1\times
s_2}(K)$ be the $K$-vector space of all matrices of size
$s_1\times s_2$ with entries in $K$. 

The {\it direct product\/} (also called {\it Kronecker product\/}) of $C_1$ and $C_2$,
denoted by $C_1\,\underline{\otimes}\,C_2$, is defined to be the linear
code consisting of all $s_1\times s_2$ matrices in which the rows
belong to $C_2$
and the columns to $C_1$; see \cite[p.~44]{tsfasman}. The direct
product codes 
usually have
poor minimum distance but are easy to decode and can be useful in
certain applications; see \cite[Chapter~18]{MacWilliams-Sloane}. 

We denote the
tensor product of $C_1$ and $C_2$---in 
the sense of multilinear algebra
\cite[p.~573]{Eisen}---by $C_1\otimes_KC_2$. As is shown 
in Lemma~\ref{jan12-15} another way to see the direct product of $C_1$ and $C_2$ is as a
tensor product.

\begin{theorem}{\cite[Theorems~2.5.2 and 2.5.3]{van-lint}}\label{dec29-14} Let 
$C_i\subset K^{s_i}$ be a linear code of length $s_i$, dimension
$k_i$, and minimum distance $\delta(C_i)$ for $i=1,2$. 
Then $C_1\,\underline{\otimes}\,C_2$ has length $s_1s_2$, dimension
$k_1k_2$, and minimum distance $\delta(C_1)\delta(C_2)$. 
\end{theorem}

\begin{theorem}{\cite[Theorem~3(d)]{wei-yang}}\label{wei-yang-th} 
Let $C_1\subset K^{s_1}$ and $C_2\subset K^{s_2}$ be
two linear codes and let $C=C_1\,\underline{\otimes}\, C_2$ be their
direct product. Then 
$$
\delta_2(C)=\min\{\delta_1(C_1)\delta_2(C_2),\delta_2(C_1)\delta_1(C_2)\}.
$$
\end{theorem}

Recall that there is a natural isomorphism ${\rm vec}\colon M_{s_1\times s_2}(K)\rightarrow K^{s_1s_2}$ of
$K$-vector spaces given by ${\rm vec}(A)=(F_1,\ldots,F_{s_1})$, where
$F_1,\ldots,F_{s_1}$ are the rows of $A$. Consider the bilinear map $\psi_0$ given by
\begin{eqnarray*}
\psi_0\colon K^{s_1}\times K^{s_2}&\longrightarrow&M_{s_1\times
s_2}(K)\\
((a_1,\ldots,a_{s_1}),(b_1,\ldots,b_{s_2}))&\longmapsto&
\left[\begin{matrix}a_1b_1&a_1b_2&\ldots&a_1b_{s_2}\cr
a_2b_1&a_2b_2&\ldots&a_2b_{s_2}\cr
\vdots&\vdots & &\vdots\cr 
a_{s_1}b_1&a_{s_1}b_2&\ldots&a_{s_1}b_{s_2}\end{matrix}\right].
\end{eqnarray*}

The Segre embedding, defined in the introduction, is given by
$\psi([a],[b])=[({\rm vec}\circ\psi_0)(a,b)]$, where
$a=(a_1,\ldots,a_{s_1})$ and $b=(b_1,\ldots,b_{s_2})$.

The next lemma is not hard to prove and probably known in some
equivalent formulation; but we could not find a reference
with the corresponding proof.

\begin{lemma}\label{jan12-15} There is an isomorphism 
$T\colon C_1\otimes_KC_2\rightarrow C_1\,\underline{\otimes}\,C_2$ of
$K$-vector spaces such that $T(a\otimes b)=\psi_0(a,b)$ for $a\in C_1$ 
and $b\in C_2$.
\end{lemma}

\begin{proof} We set $k_i=\dim_K(C_i)$ for $i=1,2$. Using the universal
property of the tensor product \cite[p.~573]{Eisen}, we get that the
bilinear map $\psi_0$ induces a linear map
\begin{eqnarray*}
T\colon C_1\otimes_K C_2&\longrightarrow& C_1\,\underline{\otimes}\,
C_2,\,\mbox{ such that},\\  
a\otimes b&\longmapsto& \psi_0(a,b)
\end{eqnarray*}
for $a\in C_1$ and
$b\in C_2$. By \cite[Formula~5, p.~267]{Mats} and Theorem~\ref{dec29-14}, one has
that $C_1\otimes_K C_2$ and $C_1\,\underline{\otimes}\,C_2$ have
dimension $k_1k_2$. Thus to prove that $T$ is an isomorphism it
suffices to prove that $T$ is a one-to-one linear map. Fix bases 
$\{\alpha_1,\ldots,\alpha_{k_1}\}$ and
$\{\beta_1,\ldots,\beta_{k_2}\}$ of $C_1$ and $C_2$, respectively.
Take any element $\gamma$ in the kernel of $T$. We can write
$$
\gamma=\sum\lambda_{i,j}\alpha_i\otimes\beta_j
$$
with $\lambda_{i,j}$ in $K$ for all $i,j$. Then 
\begin{eqnarray*}
T(\gamma)&=&\lambda_{1,1}T(\alpha_1\otimes\beta_1)+\cdots+
\lambda_{1,k_2}T(\alpha_1\otimes\beta_{k_2})+\\
& &
\lambda_{2,1}T(\alpha_2\otimes\beta_1)+\cdots+\lambda_{2,k_2}T(\alpha_2\otimes\beta_{k_2})+\\
&&\ \ \ \ \ \ \ \ \ \ \ \ \ \ \ \ \ \ \ \ \ \ \ \vdots \\
& &
\lambda_{k_1,1}T(\alpha_{k_1}\otimes\beta_1)+\cdots+\lambda_{k_1,k_2}T(\alpha_{k_1}\otimes\beta_{k_2}).
\end{eqnarray*}
Setting $\alpha_i=(\alpha_{i,1},\ldots,\alpha_{i,s_1})$, 
$\beta_j=(\beta_{j,1},\ldots,\beta_{j,s_2})$ for $i=1,\ldots,k_1$, 
$j=1,\ldots,k_2$, we get 
$$
T(\gamma)=\left[
\begin{matrix}
(\lambda_{1,1}\alpha_{1,1}\beta_1+\cdots+\lambda_{1,k_2}\alpha_{1,1}\beta_{k_2})
+\cdots+(\lambda_{k_1,1}\alpha_{k_1,1}\beta_1+\cdots+\lambda_{k_1,k_2}\alpha_{k_1,1}\beta_{k_2})\cr
(\lambda_{1,1}\alpha_{1,2}\beta_1+\cdots+\lambda_{1,k_2}\alpha_{1,2}\beta_{k_2})
+\cdots+(\lambda_{k_1,1}\alpha_{k_1,2}\beta_1+\cdots+\lambda_{k_1,k_2}\alpha_{k_1,2}\beta_{k_2})
\cr
\vdots\cr 
(\lambda_{1,1}\alpha_{1,s_1}\beta_1+\cdots+\lambda_{1,k_2}\alpha_{1,s_1}\beta_{k_2})
+\cdots+(\lambda_{k_1,1}\alpha_{k_1,s_1}\beta_1+\cdots+\lambda_{k_1,k_2}\alpha_{k_1,s_1}\beta_{k_2})
\end{matrix}
\right].
$$
Since $T(\gamma)=(0)$, using that the $\beta_i$'s are linearly independent, we get
$$
\lambda_{1,j}\alpha_1^\top+\cdots+\lambda_{k_1,j}\alpha_{k_1}^\top=0\mbox{
for }j=1,\ldots,k_2.
$$
Thus $\lambda_{i,j}=0$ for all $i,j$ and $\gamma=0$.
\end{proof}

\paragraph{\bf Hilbert functions} 

Let $K$ be a field. Recall that the {\it projective space\/} of 
dimension $s-1$ over $K$, denoted by 
$\mathbb{P}^{s-1}$, is the quotient space 
$$(K^{s}\setminus\{0\})/\sim $$
where two points $\alpha$, $\beta$ in $K^{s}\setminus\{0\}$ 
are equivalent under $\sim$ if $\alpha=\lambda{\beta}$ for some $\lambda\in K^*$. We
denote the  equivalence class of $\alpha$ by $[\alpha]$.

Let $X\neq\emptyset$ be a subset of $\mathbb{P}^{s-1}$. Consider a graded polynomial ring
$S=K[t_1,\ldots,t_s]$, over the field $K$, 
where each $t_i$ is homogeneous of degree one. Let $S_d$ denote the
set of homogeneous polynomials of 
total degree $d$ in $S$, together with the zero polynomial, and let 
$I(X)$ be the {\it vanishing ideal\/} of $X$ generated by the
homogeneous polynomials of $S$ that vanish at all points of $X$. 
The set $S_d$ is a $K$-vector space of dimension
$\binom{d+s-1}{s-1}$. We let 
$$
I(X)_d:=I(X)\cap S_d,
$$
denote the set of homogeneous polynomials in $I(X)$ of total degree $d$,
together with the zero polynomial. 
Note that $I(X)_d$ is a vector
subspace of $S_d$. The {\it Hilbert function\/} of the quotient
ring $S/I(X)$, denoted by $H_X(d)$, is defined as
$$ 
H_X(d):=\dim_K(S_d/I(X)_d). 
$$

According to a classical result of Hilbert
\cite[Theorem~4.1.3]{BHer}, there is a unique polynomial 
\begin{eqnarray*}  
h_X(t)=c_kt^k+(\mbox{terms of lower degree})
\end{eqnarray*}
of degree $k\geq 0$, with rational coefficients, such that $h_X(d)=H_X(d)$
for $d\gg 0$. The integer $k+1$ is the {\it Krull dimension\/} of 
$S/I(X)$, $k$ is the {\it dimension\/} of $X$, and 
$h_X(t)$ is the {\it Hilbert
polynomial\/} 
of $S/I(X)$. The positive integer $c_{k}(k!)$ is the {\it degree\/}
of $S/I(X)$. 
The {\it index of regularity\/}
of $S/I(X)$, denoted by  
${\rm reg}(S/I(X))$, is the least integer $r\geq 0$ such that
$h_X(d)=H_X(d)$ for $d\geq r$. The degree and the Krull dimension are
denoted  by ${\rm deg}(S/I(X))$ and $\dim(S/I(X))$, respectively. 

\begin{proposition}{\rm(\cite{duursma-renteria-tapia},
\cite{geramita-cayley-bacharach}, 
\cite{degree-lattice})}\label{hilbert-function-dim=1} 
If $X$ is a finite set and $r={\rm reg}(S/I(X))$, then 
$$
1=H_X(0)<H_X(1)<\cdots<H_X(r-1)<H_X(d)=\deg(S/I(X))=|X|\hspace{.5cm} {\it for}
\ d\geq r.
$$
\end{proposition}

\paragraph{\bf Projective Reed-Muller-type codes} In this part we introduce
the family of projective Reed-Muller-type codes and its connection to vanishing
ideals and Hilbert functions. 

Let $K=\mathbb{F}_q$ be a finite field and let $X=\{P_1,\ldots,P_m\}\neq\emptyset$ be a
subset of $\mathbb{P}^{s-1}$ with $m=|X|$. 
Fix a degree $d\geq 0$. For each $i$ there is $f_i\in S_d$ such that
$f_i(P_i)\neq 0$; we refer to Section~\ref{section-segre-codes} to see
a convenient way to choose  $f_1,\ldots,f_m$. There is a well-defined $K$-linear map given by  
\begin{equation}\label{ev-map}
{\rm ev}_d\colon S_d=K[t_1,\ldots,t_s]_d\rightarrow K^{|X|},\ \ \ \ \ 
f\mapsto
\left(\frac{f(P_1)}{f_1(P_1)},\ldots,\frac{f(P_m)}{f_m(P_m)}\right).
\end{equation}

The map ${\rm ev}_d$ is called an {\it evaluation map}. The image of 
$S_d$ under ${\rm ev}_d$, denoted by  $C_X(d)$, is called a {\it
projective Reed-Muller-type code\/} of degree $d$ over the set $X$ \cite{duursma-renteria-tapia,GRT}. It is
also called an {\it evaluation code\/} associated to $X$
\cite{gold-little-schenck}. The kernel of the
evaluation map ${\rm ev}_d$ is $I(X)_d$. Hence there is an isomorphism of $K$-vector spaces
$S_d/I(X)_d\simeq C_X(d)$.

\begin{definition}\label{segre-code-def}
If $\mathbb{X}$ is the Segre product of
$\mathbb{X}_1$ and $\mathbb{X}_2$, we say that $C_{\mathbb{X}}(d)$ is a
{\it projective Segre code\/} of degree  $d$; recall that $\mathbb{X}$
is the image of $\mathbb{X}_1\times\mathbb{X}_2$ under the Segre
embedding $\psi$.
\end{definition}

\begin{definition}\label{basic-parameters-def} The {\it basic parameters} of the linear
code $C_X(d)$ are:
\begin{itemize}
\item[(a)] {\it length\/}: $|X|$,

\item[(b)] {\it dimension\/}: $\dim_K C_X(d)$, and 

\item[(c)] {\it minimum distance\/}: $\delta(C_X(d))$. We also denote $\delta(C_X(d))$ simply by
$\delta_X(d)$.
\end{itemize}
\end{definition}

The basic parameters of projective Reed-Muller-type codes have been computed in a number
of cases. If $X=\mathbb{P}^{s-1}$ then $C_X(d)$ is the {\it classical projective
Reed--Muller code\/} and its basic parameters are described in
\cite[Theorem~1]{sorensen}. If $X$ is a projective torus, $C_X(d)$ is
the {\it generalized projective 
Reed--Solomon code\/} and its basic parameters are described in
\cite[Theorem~3.5]{ci-codes}. 

The following summarizes the well-known relation between 
projective Reed-Muller-type codes and the theory of Hilbert functions.

\begin{proposition}{\rm(\cite{GRT}, \cite{algcodes})}\label{jan4-15}
The following hold. 
\begin{itemize}
\item[{\rm (i)}] $H_X(d)=\dim_KC_X(d)$ for $d\geq 0$.
\item[{\rm (ii)}] $\delta_X(d)=1$ for $d\geq {\rm reg}(S/I(X))$.  
\item[{\rm (iii)}] $S/I(X)$ is a Cohen--Macaulay graded ring of dimension
$1$. 
\item[{\rm (iv)}] $C_X(d)\neq(0)$ for $d\geq 0$. 
\end{itemize}
\end{proposition} 

\paragraph{\bf Segre products} To avoid repetitions, we continue to employ
the notations and definitions used in Section~\ref{section-intro}. For
the rest of this section $K$ will denote an arbitrary field.

A {\it standard algebra\/} over a field $K$ is a 
finitely generated graded $K$-algebra 
$A=\bigoplus_{d=0}^{\infty}A_d$ such that $A=K[A_1]$ and $A_0=K$ (that is, $A$ is
isomorphic to $K[\mathbf{x}]/I$, for some polynomial ring
$K[\mathbf{x}]$ with the 
standard grading and for some graded ideal $I$).

\begin{definition}{\cite[p.~304]{Eisen}}\label{segre-product-def}\rm\
Let $A=\oplus_{d\geq 0}A_d$, $B=\oplus_{d\geq 0}B_d$ be 
two standard algebras over a field 
$K$. The {\it Segre product\/} of $A$ and $B$, denoted by
$A\otimes_\mathcal{S}B$, is the graded algebra 
$$
A\otimes_{\mathcal S}B:=
(A_0\otimes_K B_0)\oplus(A_1\otimes_K B_1)\oplus\cdots 
\subset A\otimes_K B,
$$
with the normalized grading $(A\otimes_{\mathcal S}B)_d:=A_d\otimes_K B_d$ for $d\geq 0$. 
The tensor product algebra $A\otimes_K B$ is graded by
$$(A\otimes_K B)_p:=\sum_{i+j=p} A_i\otimes_K B_j.
$$ 
\end{definition}

\begin{example}{\cite[p.~161]{BG-book}} The Segre product (resp.
tensor product) of $K[\mathbf{x}]$ and 
$K[\mathbf{y}]$ is $$K[\mathbf{x}]\otimes_\mathcal{S} K[\mathbf{y}]\simeq 
K[\{x_iy_j\vert\, 1\leq i\leq a_1,\, 1\leq j\leq a_2\}]$$ 
(resp. $K[\mathbf{x}]\otimes_KK[\mathbf{y}]\simeq
K[\mathbf{x},\mathbf{y}]$). Notice that the elements $x_iy_j$ have
normalized degree $1$ as elements of $K[\mathbf{x}]\otimes_\mathcal{S}
K[\mathbf{y}]$ and total degree $2$ as elements of
$K[\mathbf{x}]\otimes_K K[\mathbf{y}]$.
\end{example}

The next result is well-known assuming that $\mathbb{X}_1$ and 
$\mathbb{X}_2$ are projective
algebraic sets; see for instance \cite[Exercise 13.14(d)]{Eisen}.
However David Eisenbud pointed out to us that the result is valid in
general. We give a proof of the general case.

\begin{theorem}\label{dec4-14-1} Let $K$ be a field. If $\mathbb{X}_1$, $\mathbb{X}_2$ are subsets
of the projective spaces  $\mathbb{P}^{a_1-1}$, $\mathbb{P}^{a_2-1}$,
respectively, and $\mathbb{X}$ is the Segre product of $\mathbb{X}_1$
and $\mathbb{X}_2$,  
then the following hold:
\begin{itemize}
\item[(a)]
$(K[\mathbf{x}]/I(\mathbb{X}_1))_d\otimes_K(K[\mathbf{y}]/I(\mathbb{X}_2))_d
\simeq(K[\mathbf{t}]/I(\mathbb{X}))_d$ as $K$-vector spaces for $d\geq
0$. 
\item[(b)] $K[\mathbf{x}]/I(\mathbb{X}_1)\otimes_{\mathcal{S}}K[\mathbf{y}]/I(\mathbb{X}_2)
\simeq K[\mathbf{t}]/I(\mathbb{X})$ as standard graded algebras. 
\item[(c)] $H_{\mathbb{X}_1}(d)H_{\mathbb{X}_2}(d)=H_{\mathbb{X}}(d)$
for $d\geq 0$.
\item[(d)] ${\rm reg}(K[\mathbf{t}]/I(\mathbb{X}))=\max\{{\rm
reg}(K[\mathbf{x}]/I(\mathbb{X}_1)), {\rm
reg}(K[\mathbf{y}]/I(\mathbb{X}_2))\}$.
\item[(e)] If $\rho_1=\dim(K[\mathbf{x}]/I(\mathbb{X}_1))$ and 
$\rho_2=\dim(K[\mathbf{y}]/I(\mathbb{X}_2))$, then
$$
\deg(K[\mathbf{t}]/I(\mathbb{X}))=
\deg(K[\mathbf{x}]/I(\mathbb{X}_1))\deg(K[\mathbf{y}]/I(\mathbb{X}_2))
\binom{\rho_1+\rho_2-2}{\rho_1-1}.
$$
\end{itemize}
\end{theorem}

\begin{proof} (a): Let $\sigma$ be the epimorphism of $K$-algebras $\sigma\colon
K[\mathbf{t}]\rightarrow K[\{x_iy_j\vert\, i\in[\![1,a_1]\!],\,
j\in[\![1,a_2]\!]\}]$ induced by $t_{ij}\mapsto x_iy_j$, where
$[\![1,a_i]\!]=\{1,\ldots,a_i\}$. For each
$x^by^c$ with $\deg(x^b)=\deg(y^c)=d$ there is a unique $t^a\in
K[\mathbf{t}]_d$ such that $t^a=t_{i_1,j_1}\cdots t_{i_d,j_d}$,
$1\leq i_1\leq\cdots\leq i_d$, $1\leq j_1\leq\cdots\leq j_d$ and
$\sigma(t^a)=x^by^c$. Notice that if $\sigma(t^\alpha)=x^by^c$ for some
other monomial $t^\alpha\in K[\mathbf{t}]_d$, then $t^a-t^\alpha\in
I(\mathbb{X})$. This is used below to ensure that the mapping of
Eq.~(\ref{dec4-14}) is surjective. Setting $\varphi_0(x^b,y^c)=t^{a}$,
we get a $K$-bilinear map
$$ 
\varphi_0\colon K[\mathbf{x}]_d\times K[\mathbf{y}]_d\rightarrow
K[\mathbf{t}]_d
$$
induced by $\varphi_0(x^b,y^c)=t^a$. Notice that
$\varphi_0(\sum\lambda_kx^{b_k},\sum\mu_\ell y^{c_\ell})=
\sum\lambda_k\mu_\ell\varphi_0(x^{b_k},y^{c_\ell})$, where the $\lambda_k$'s and 
$\mu_\ell$'s are in $K$. To show that $\varphi_0$ induces
a $K$-bilinear map 
\begin{equation}\label{dec4-14}
\varphi\colon (K[\mathbf{x}]_d/I(\mathbb{X}_1)_d)\times
(K[\mathbf{y}]_d/I(\mathbb{X}_2)_d)\rightarrow
K[\mathbf{t}]_d/I(\mathbb{X})_d,\ \
(\overline{x^b},\overline{y^c})\mapsto \overline{\varphi_0(x^b,y^c)},
\end{equation}
which is a surjection, it suffices to show that for any 
$f\in K[\mathbf{x}]_d$ that vanish on $\mathbb{X}_1$ (resp. $g\in
K[\mathbf{y}]_d$ that vanish on $\mathbb{X}_2$) one has that
$\varphi_0(f,g)$ vanishes at all points of $\mathbb{X}$. Assume that
$f=\lambda_1x^{b_1}+\cdots+\lambda_mx^{b_m}$ is a polynomial in $K[\mathbf{x}]_d$ that
vanish on $\mathbb{X}_1$ and that
$g=\mu_1y^{c_1}+\cdots+\mu_ry^{c_r}$ is a polynomial in $K[\mathbf{y}]_d$ with $\lambda_k$, $\mu_\ell$ in $K$ for
all $k,\ell$. For each $x^{b_k}y^{c_\ell}$ there is $t^{a_{k\ell}}\in K[\mathbf{t}]$ such that
$\sigma(t^{a_{k\ell}})=x^{b_k}y^{c_\ell}$. Then 
\begin{eqnarray*}
\varphi_0(f,g)&=&\sum\lambda_k\mu_\ell\varphi_0(x^{b_k},y^{c_\ell})=
\sum\lambda_k\mu_\ell t^{a_{k\ell}},
\mbox{ and}\\
\varphi_0(f,g)(x_iy_j)&=&(\lambda_1x^{b_1}+\cdots+\lambda_mx^{b_m})(\mu_1y^{c_1}+\cdots+\mu_ry^{c_r}), 
\end{eqnarray*}
where we use $(x_iy_j)$ as a short hand for
$(x_1y_1,x_1y_2,\ldots,x_1y_{a_2},
\ldots,x_{a_1}y_1,x_{a_1}y_2,\ldots,x_{a_1}y_{a_2})$. Now if
$[(\alpha_1,\ldots,\alpha_{a_1})]$ is in $\mathbb{X}_1$ and 
$[(\beta_1,\ldots,\beta_{a_2})]$ is in $\mathbb{X}_2$, making $x_i=
\alpha_i$ and $y_j=\beta_j$ for all $i,j$ in the last equality, we get
$\varphi_0(f,g)(\alpha_i\beta_j)=0$. Therefore, by the universal property of the
tensor product \cite[p.~573]{Eisen}, there is a surjective map 
$\overline{\varphi}$ that makes the following diagram commutative: 

$$
\setlength{\unitlength}{.040cm}
\begin{picture}(20,10)
\put(-150,0){$(K[\mathbf{x}]_d/I(\mathbb{X}_1)_d)\times
(K[\mathbf{y}]_d/I(\mathbb{X}_2)_d)$}
\put(27,0){$(K[\mathbf{x}]_d/I(\mathbb{X}_1)_d)\otimes_K
(K[\mathbf{y}]_d/I(\mathbb{X}_2)_d)$}
\put(-3,3){\vector(1,0){20}}
\put(3,8){$\phi$}
\put(-79,-5){\vector(0,-1){20}}
\put(-77,-15){$\varphi$}
\put(-105,-37){$K[\mathbf{t}]_d/I(\mathbb{X})_d$}
\put(10,-22){$\overline{\varphi}$}
\put(90,-7){\vector(-4,-1){100}}
\end{picture}
$$
\vspace{1.3cm} 

\noindent where $\phi$ is the canonical 
map, given by
$\phi(\overline{f},\overline{g})=\overline{f}\otimes\overline{g}$,
 and $\varphi=\overline{\varphi}\phi$. 

For each $t^\alpha\in K[\mathbf{t}]_d$ let $x^{b}\in
K[\mathbf{x}]_d$ and $y^c\in K[\mathbf{y}]_d$ be such that
$\sigma(t^{\alpha})=x^{b}y^{c}$. We set $\sigma_1(t^{\alpha})=x^b$ and
$\sigma_2(t^{\alpha})=y^c$. 
Thus we have
a surjective $K$-linear map 
$$
\sigma_0^*\colon K[\mathbf{t}]_d\rightarrow 
K[\mathbf{x}]_d/I(\mathbb{X}_1)_d\otimes_K
K[\mathbf{y}]_d/I(\mathbb{X}_2)_d
$$
given by $\sigma_0^*(\sum\lambda_\alpha t^\alpha)=
\sum\lambda_\alpha\overline{\sigma_1(t^\alpha)}\otimes\overline{\sigma_2(t^\alpha)}$, where the
$\lambda_\alpha$'s are in $K$. Notice that the $K$-vector space on the right
hand side is generated by all $\overline{x^b}\otimes \overline{y^c}$ 
such that $x^b\in K[\mathbf{x}]_d$ and $y^c\in
K[\mathbf{y}]_d$. Take $f\in I(\mathbb{X})_d$, then
$\sigma(f)(\alpha_i\beta_j)=0$ for all
$\alpha=[(\alpha_1,\ldots,\alpha_{a_1})]\in\mathbb{X}_1$ and all
$\beta=[(\beta_1,\ldots,\beta_{a_2})]\in\mathbb{X}_2$. We can write
$\sigma(f)=\sum_{\ell=1}^kf_\ell g_\ell$ with $f_\ell\in K[\mathbf{x}]_d$, 
$g_\ell\in K[\mathbf{y}]_d$ for $\ell=1,\ldots,k$, and
$\sigma_0^*(f)=\sum_{\ell=1}^k\overline{f_\ell}\otimes\overline{g_\ell}$. Next
we show, by induction on $k$, that $\sigma_0^*(f)=0$, i.e., $f\in{\rm
ker}(\sigma_0^*)$. If $k=1$, we may assume that $f_1\notin
I(\mathbb{X}_1)$ otherwise $\overline{f_1}=\overline{0}$. Pick
$\alpha\in\mathbb{X}_1$ such that $f_1(\alpha)\neq 0$. Then, as
$f_1(\alpha)g_1(\beta)=0$ for all $\beta\in\mathbb{X}_2$, one has
$g_1\in I(\mathbb{X}_2)$ and $\overline{g_1}=\overline{0}$. We may now
assume that $k>1$ and $\overline{f_k}\neq 0$. Pick
$\alpha\in\mathbb{X}_1$ such that $f_k(\alpha)\neq 0$. By hypothesis
the polynomial
$$  
f_1(\alpha)g_1+\cdots+f_k(\alpha)g_k
$$
is in $K[\mathbf{y}]_d$ and vanishes at all points of $\mathbb{X}_2$.
Thus
$$
\overline{g_k}=-(f_1(\alpha)/f_k(\alpha))\overline{g}_1-
\cdots-(f_{k-1}(\alpha)/f_k(\alpha))\overline{g}_{k-1}.
$$
Therefore, setting $h_\ell=f_\ell-(f_\ell(\alpha)/f_k(\alpha))f_k$ for
$\ell=1,\ldots,k-1$, we get 
$$ 
\sigma_0^*(f)=\sum_{\ell=1}^k\overline{f_\ell}\otimes\overline{g}_\ell=\sum_{\ell=1}^{k-1}
\overline{h}_\ell\otimes\overline{g}_\ell
$$
and $\sum_{\ell=1}^{k-1}h_\ell(\gamma)g_\ell(\beta)=0$ for all
$\gamma\in\mathbb{X}_1$ and $\beta\in\mathbb{X}_2$. Thus, by
induction, $\sigma_0^*(f)=0$. Hence $I(\mathbb{X})_d \subset{\rm
ker}(\sigma_0^*)$. Therefore $\sigma_0^*$ induces a $K$-linear
surjection
$$
\sigma^*\colon K[\mathbf{t}]_d/I(\mathbb{X})_d\rightarrow 
(K[\mathbf{x}]_d/I(\mathbb{X}_1)_d)\otimes_K
(K[\mathbf{y}]_d/I(\mathbb{X}_2)_d).
$$
Altogether we get that the linear maps $\overline{\varphi}$ and $\sigma^*$ are bijective.  

Items (b) to (e) follow directly from (a) and its proof.
\end{proof}

\section{Basic parameters of projective Segre codes}\label{section-segre-codes}
In this section we study projective Segre codes and their basic
parameters; including the second generalized Hamming weight. It is
shown that direct product codes of projective Reed-Muller-type codes 
are projective Segre codes.  Then some applications are given. 
We continue to employ the notations and definitions used in
Sections~\ref{section-intro} and \ref{prelim-section}. 

In preparation for our main theorem, let $K=\mathbb{F}_q$ be a finite field, let $a_1,a_2$ be two 
positive integers with $a_1\geq a_2$, and for $i=1,2$, let $\mathbb{X}_i$ be 
a non-empty subset of the projective space $\mathbb{P}^{a_i-1}$ over
$K$. 
We set $s=a_1a_2$ and $s_i=|\mathbb{X}_i|$ for $i=1,2$. 
The {\it Segre embedding\/} is given by
\begin{eqnarray*}
\psi\colon\mathbb{P}^{a_1-1}\times\mathbb{P}^{a_2-1}&\rightarrow&\mathbb{P}^{a_1a_2-1}=\mathbb{P}^{s-1}\\
([(\alpha_1,\ldots,\alpha_{a_1})],[(\beta_1,\ldots,\beta_{a_2})])&\rightarrow&[(\alpha_1\beta_1,\alpha_1\beta_2,\ldots,\alpha_1\beta_{a_2},\\
& &\ \ \alpha_2\beta_1,\alpha_2\beta_2,\ldots,\alpha_2\beta_{a_2},\\
& &\ \ \ \  \ \ \ \ \ \ \  \vdots\\
&
&\ \
\alpha_{a_1}\beta_1,\alpha_{a_1}\beta_2,\ldots,\alpha_{a_1}\beta_{a_2})].
\end{eqnarray*}

The image of $\mathbb{X}_1\times \mathbb{X}_2$ under the map $\psi$,
denoted by $\mathbb{X}$, is the {\it Segre product\/} of $\mathbb{X}_1$ and $\mathbb{X}_2$. 
As $\psi$ is injective, we get 
$|\mathbb{X}|=|\mathbb{X}_1||\mathbb{X}_2|=s_1s_2$. Then we can write
$\mathbb{X}$, $\mathbb{X}_1$, and $\mathbb{X}_2$ as:
\begin{eqnarray*}
\mathbb{X}=\{P_{1,1},\ldots,P_{s_1,s_2}\}&
=&\{P_{1,1},\ P_{1,2},\ldots,\ P_{1,s_2},\\
& &\ \, P_{2,1},\ P_{2,2},\ldots,\ P_{2,s_2},\\
& &\ \ \ \  \ \ \ \ \ \ \  \vdots\\
&
&\ 
P_{s_1,1},P_{s_1,2},\ldots,P_{s_1,s_2}\},
\end{eqnarray*}
$\mathbb{X}_1=\{Q_1,\ldots,Q_{s_1}\}$, and
$\mathbb{X}_2=\{R_1,\ldots,R_{s_2}\}$, respectively, where 
$$
Q_i=[(\alpha_{i,1},\alpha_{i,2},\ldots,\alpha_{i,a_1})]\ \mbox{ and
}\ R_j=[(\beta_{j,1},\beta_{j,2},\ldots,\beta_{j,a_2})],
$$
for $i=1,\ldots,s_1$ and $j=1,\ldots,s_2$. Because of the embedding
$\psi$ each $P_{i,j}\in\mathbb{X}$ is of the 
form 
\begin{eqnarray*}
P_{i,j}=\psi(Q_i,R_j)&=&[(\alpha_{i,1}\cdot\beta_{j,1},
\alpha_{i,1}\cdot\beta_{j,2},\ldots,\alpha_{i,1}\cdot\beta_{j,a_2},\\
& &\ \
\alpha_{i,2}\cdot\beta_{j,1},\alpha_{i,2}\cdot\beta_{j,2},\ldots,\alpha_{i,2}\cdot\beta_{j,a_2},\\
& &\ \ \ \  \ \ \ \ \ \ \  \vdots\\
&
&\ \
\alpha_{i,a_1}\cdot\beta_{j,1},\alpha_{i,a_1}\cdot\beta_{j,2},\ldots,\alpha_{i,a_1}\cdot\beta_{j,a_2})].
\end{eqnarray*}

For use below notice that 
for each $i\in[\![1,s_1]\!]$ and for each $j\in[\![1,s_2]\!]$ 
there are $k_i\in[\![1,a_1]\!]$ and $\ell_j\in[\![1,a_2]\!]$
such that $\alpha_{i,k_i}\neq 0$ and $\beta_{j,\ell_j}\neq 0$. In
fact, choose $k_i$ to be the smallest 
$k\in[\![1,a_1]\!]$ such that $\alpha_{i,k}\neq 0$,  and choose 
$\ell_j$ to be the smallest
$\ell\in[\![1,a_2]\!]$ such that $\beta_{j,\ell}\neq 0$. Hence 
$\alpha_{i,k_i}\cdot \beta_{j,\ell_j}\neq 0$.

Setting
$K[\mathbf{t}]=K[t_{1,1},t_{1,2}\ldots,t_{1,a_2},\ldots,t_{a_1,1},t_{a_1,2},\ldots,t_{a_1,a_2}]$
and fixing an integer $d\geq 1$, define
$f_{i,j} (t_{1,1},\ldots,t_{a_1,a_2})=(t_{k_i,\ell_j})^d$. Then
$f_{i,j}(P_{i,j})=(\alpha_{i,k_i}\cdot \beta_{j,\ell_j})^d\neq 0$. The
evaluation map ${\rm ev}_d$ is defined as:
\begin{eqnarray*}
{\rm ev}_d\colon K[\mathbf{t}]_d&\rightarrow
&K^{|\mathbb{X}|}=K^{s_1s_2},\ \ \
\\ 
 f&\rightarrow&\left(\frac{f(P_{1,1})}{f_{1,1}(P_{1,1})},\frac{f(P_{1,2})}{f_{1,2}(P_{1,2})},\ldots,
\frac{f(P_{s_1,s_2})}{f_{s_1,s_2}(P_{s_1,s_2})}\right). 
\end{eqnarray*}
This is a linear map of 
$K$-vector spaces. 
The image of ${\rm ev}_d$, denoted by
$C_{\mathbb{X}}(d)$, defines a projective Reed-Muller-type linear code of
degree $d$ that we call a {\it
projective Segre code\/} of
degree $d$. 

For each $i\in [\![1,s_1]\!]$ and for each $j\in [\![1,s_2]\!]$,
define the following polynomials:
$$ 
g_i(x_1,\ldots,x_{a_1})=x_{k_i}^d\in K[x_1,\ldots,x_{a_1}]_d\ \mbox{
and }\ 
h_j(y_1,\ldots,y_{a_2})=y_{\ell_j}^d\in K[y_1,\ldots,y_{a_2}]_d.
$$
Clearly $g_i(Q_i)=\alpha_{i,k_i}^d\neq 0$, 
$h_j(R_j)=\beta_{j,\ell_j}^d\neq 0$,
$f_{i,j}(P_{i,j})=(\alpha_{i,k_i})^dh_j(R_j)=g_i(Q_i)(\beta_{j,\ell_j})^d$. We also define the following two
evaluation maps: 
\begin{eqnarray*}
{\rm ev}^1_d\colon K[x_1,\ldots,x_{a_1}]_d&\rightarrow
&K^{|\mathbb{X}_1|}=K^{s_1},\ \ \ \\ 
 g&\rightarrow&\left(\frac{g(Q_{1})}{g_{1}(Q_{1})},\frac{g(Q_{2})}{g_{2}(Q_{2})},\ldots,
\frac{g(Q_{s_1})}{g_{s_1}(Q_{s_1})}\right), \mbox{ and}\\
{\rm ev}^2_d\colon K[y_1,\ldots,y_{a_2}]_d&\rightarrow
&K^{|\mathbb{X}_2|}=K^{s_2},\ \ \ \\ 
 h&\rightarrow&\left(\frac{h(R_{1})}{h_{1}(R_{1})},\frac{h(R_{2})}{h_{2}(R_{2})},\ldots,
\frac{h(R_{s_2})}{h_{s_2}(R_{s_2})}\right),
\end{eqnarray*}
and their corresponding Reed-Muller-type linear codes $C_{\mathbb{X}_i}(d):={\rm im}({\rm
ev}^i_d)$ for $i=1,2$. 

\smallskip

Let $C$ be a linear code. From Section~\ref{prelim-section}
recall that $\delta_r(C)$ is the $r$th generalized Hamming weight of
$C$ and that $\delta_1(C)$ is the minimum distance $\delta(C)$
of $C$. 


\smallskip

We come to the main result of this section. 
\begin{theorem}\label{azucena-maria-vila} Let $K=\mathbb{F}_q$ be a
finite field, let $\mathbb{X}_i\subset\mathbb{P}^{a_i-1}$ for 
$i=1,2$, and let $\mathbb{X}$ be the Segre product of $\mathbb{X}_1$ and
$\mathbb{X}_2$. The following hold.
\begin{itemize}
\item[(a)] 
$|\mathbb{X}|=|\mathbb{X}_1||\mathbb{X}_2|$.
\item[(b)] 
$\dim_K(C_{\mathbb{X}}(d))=\dim_K(C_{\mathbb{X}_1}(d))\dim_K(C_{\mathbb{X}_2}(d))$
for $d\geq 1$.
\item[(c)] $C_{\mathbb{X}}(d)$ is the direct
product $C_{\mathbb{X}_1}(d)\,\underline{\otimes}\,C_{\mathbb{X}_2}(d)$ of 
$C_{\mathbb{X}_1}(d)$ and $C_{\mathbb{X}_2}(d)$ for $d\geq 1$.
\item[(d)] 
$\delta({C_\mathbb{X}}(d))=\delta({C_{\mathbb{X}_1}}(d))\delta({C_{\mathbb{X}_2}}(d))$
for $d\geq 1$.
\item[(e)]  $\delta_2(C_{\mathbb{X}}(d))=\min\{\delta_1(C_{\mathbb{X}_1}(d))\delta_2(C_{\mathbb{X}_2}(d)),
\delta_2(C_{\mathbb{X}_1}(d))\delta_1(C_{\mathbb{X}_2}(d))\}$ for
$d\geq 1$.
\item[(f)] $\delta(C_\mathbb{X}(d))=1$ for $d\geq\max\{{\rm
reg}(K[\mathbf{x}]/I(\mathbb{X}_1)), {\rm
reg}(K[\mathbf{y}]/I(\mathbb{X}_2))\}$.
\end{itemize}
\end{theorem}
\begin{proof} (a): This is clear because the Segre embedding is a
one-to-one map. 

(b): Since $K[\mathbf{x}]_d/I(\mathbb{X}_1)_d\simeq
C_{\mathbb{X}_1}(d)$, $K[\mathbf{y}]_d/I(\mathbb{X}_2)_d\simeq
C_{\mathbb{X}_2}(d)$, and $K[\mathbf{t}]_d/I(\mathbb{X})_d\simeq
C_{\mathbb{X}}(d)$, the results follows at once from 
Theorem~\ref{dec4-14-1}. 

(c): Given $f\in K[\mathbf{t}]_d$, the entries of ${\rm ev}_d(f)$
can be arranged as:
\begin{eqnarray}\label{k-product-iso}
{\rm ev}_d(f)&=
&\left(\frac{f(P_{1,1})}{f_{1,1}(P_{1,1})},\ \
\frac{f(P_{1,2})}{f_{1,2}(P_{1,2})},\ldots,\ \ 
\frac{f(P_{1,s_2})}{f_{1,s_2}(P_{1,s_2})}\right.,\ \ \ \ \ \rightarrow
\Gamma_1\\ 
& &\ \ \
\frac{f(P_{2,1})}{f_{2,1}(P_{2,1})},\ \
\frac{f(P_{2,2})}{f_{2,2}(P_{2,2})},\ldots,\ \ 
\frac{f(P_{2,s_2})}{f_{2,s_2}(P_{2,s_2})},\ \ \ \ \ \rightarrow
\Gamma_2\nonumber\\
& &\ \ \ \ \ \ \ \ \ \  \vdots\ \ \ \ \ \ \ \ \ \ \ \ \ \vdots\ \ \ \ \ \
\ \ \ \ \ \ \ \ \ \ \ \ \ \ \vdots\ \ \ \ \ \ \ \ \ \ \ \ \ \ \ \ \vdots\nonumber \\
& &\ \ \
\left.\frac{f(P_{s_1,1})}{f_{s_1,1}(P_{s_1,1})},\frac{f(P_{s_1,2})}{f_{s_1,2}(P_{s_1,2})},\ldots,
\frac{f(P_{s_1,s_2})}{f_{s_1,s_2}(P_{s_1,s_2})}\right)\rightarrow
\Gamma_{s_1}\nonumber\\ 
& & \ \ \ \ \ \ \ \ \ \ \downarrow\ \ \ \ \ \ \ \ \ \  \ \ \ \downarrow
\ \ \ \ \ \ \ \ \ \ \ \ \ \ \ \ \ \ \downarrow\nonumber\\ 
& & \ \ \ \ \ \ \ \ \ \ \Lambda_1\ \ \ \ \ \ \ \ \ \ \ \Lambda_2\ \ \ \ \ \cdots 
\ \ \ \  \ \ \Lambda_{s_2}\nonumber 
\end{eqnarray}
where $\Gamma_1,\ldots,\Gamma_{s_1}$ and
$\Lambda_1,\ldots,\Lambda_{s_2}$ are row and column vectors,
respectively. Thus ${\rm ev}_d(f)$ can be viewed as a matrix of size
$s_1\times s_2$. 
Next we show that $\Gamma_i\in C_{\mathbb{X}_2}(d)$ and
$\Lambda_j^\top\in C_{\mathbb{X}_1}(d)$ for all $i,j$.  
Define the following polynomials
\begin{eqnarray*}
h_{Q_i}&=& f(\alpha_{i,1}\cdot y_1,\alpha_{i,1}\cdot
y_2,\ldots,\alpha_{i,1}\cdot y_{a_2},\\ 
& &\ \ \ \alpha_{i,2}\cdot y_1,\alpha_{i,2}\cdot
y_2,\ldots,\alpha_{i,2}\cdot y_{a_2},\\ 
& &\ \ \ \ \ \ \ \ \ \ \ \ \ \ \ \ \ \ \ \ \ \ \vdots \\ 
& &\ \ \  \alpha_{i,a_1}\cdot y_1,\alpha_{i,a_1}\cdot
y_2,\ldots,\alpha_{i,a_1}\cdot y_{a_2})\in
K[y_1,\ldots,y_{a_2}]_d,\mbox{ and}\\ 
& & \\ 
g_{R_j}&=&
f(x_1\cdot\beta_{j,1},x_1\cdot\beta_{j,2},\ldots,x_1\cdot\beta_{j,a_2},\\ 
& &\ \ \ x_2\cdot\beta_{j,1},x_2\cdot\beta_{j,2}
,\ldots,x_2\cdot\beta_{j,a_2},\\ 
& &\ \ \ \ \ \ \ \ \ \ \ \ \ \ \ \ \ \ \ \ \ \ \vdots \\ 
& &\ \ \  x_{a_1}\cdot\beta_{j,1},x_{a_1}\cdot\beta_{j,2}
,\ldots,x_{a_1}\cdot\beta_{j,a_2} )\in K[x_1,\ldots,x_{a_1}]_d.
\end{eqnarray*}
Observe that $f(P_{ij})=h_{Q_i}(R_j)=g_{R_j}(Q_i)$.  Noticing the equalities 
\begin{eqnarray*}
\Gamma_i&=
&\left(\frac{f(P_{i1})}{f_{i1}(P_{i1})},\frac{f(P_{i2})}{f_{i2}(P_{i2})},\ldots,
\frac{f(P_{is_2})}{f_{is_2}(P_{is_2})}\right)=\\ 
& &\ \ \  \left(\frac{h_{Q_i}(R_1)}{\alpha_{i,k_i}^d\cdot
h_1(R_1)},\frac{h_{Q_i}(R_2)}{\alpha_{i,k_i}^d\cdot h_2(R_2)},\ldots, 
\frac{h_{Q_i}(R_{s_2})}{\alpha_{i,k_i}^d\cdot
h_{s_2}(R_{s_2})}\right)=\frac{1}{(\alpha_{i,k_i})^d}\cdot {\rm
ev}_d^2(h_{Q_i}),\\
\Lambda_j^\top&=&\frac{1}{(\beta_{j,\ell_j})^d}\cdot {\rm
ev}_d^1(g_{R_j}),
\end{eqnarray*}
for $i=1,\ldots,s_1$ and $j=1,\ldots,s_2$, we get that $\Gamma_i\in C_{\mathbb{X}_2}(d)$ and
$\Lambda_j^\top\in C_{\mathbb{X}_1}(d)$ for all $i,j$. This proves
that $C_{\mathbb{X}}(d)$ can be regarded as a linear subspace
of $C_{\mathbb{X}_1}(d)\,\underline{\otimes}\, C_{\mathbb{X}_2}(d)$. 
By part (b) and Theorem~\ref{dec29-14} the linear codes $C_{\mathbb{X}}(d)$ and 
$C_{\mathbb{X}_1}(d)\,\underline{\otimes}\, C_{\mathbb{X}_2}(d)$ have
the same dimension. Hence these linear spaces must be equal. 

(d): From Theorem~\ref{dec29-14} and part (c), one has 
$\delta({C_\mathbb{X}}(d))=\delta({C_{\mathbb{X}_1}}(d))\delta({C_{\mathbb{X}_2}}(d))$
for $d\geq 1$.

(e): It follows at once from
Theorem~\ref{wei-yang-th} and part (c). 

(f): This follows from Proposition~\ref{jan4-15}(ii) and
Theorem~\ref{dec4-14-1}(d). 
\end{proof}

\begin{remark} 
This result tells us that the direct product of
projective Reed-Muller-type codes is again a projective Reed-Muller-type code. 
\end{remark}

\begin{definition} If $K^*=K\setminus\{0\}$
and $\mathbb{X}_i$ is the image of $(\mathbb{K}^*)^{a_i}$, under the 
map $(K^*)^{a_i}\rightarrow\mathbb{P}^{a_i-1}$, $x\rightarrow [x]$, we call
$\mathbb{X}_i$ a {\it projective torus\/} in $\mathbb{P}^{a_i-1}$.   
\end{definition}

Our main theorem gives a wide generalization of most of the main results 
of \cite{GR,G-SR-M,GRH,GRT}.

\begin{remark} If $\mathbb{X}_1=\mathbb{P}^{a_1-1}$ 
and $\mathbb{X}_2=\mathbb{P}^{a_2-1}$, using
Theorem~\ref{azucena-maria-vila} we recover the formula for the
minimum distance of $C_\mathbb{X}(d)$ given in 
\cite[Theorem~5.1]{GRT}, and if $\mathbb{X}_i$ is a projective torus
for $i=1,2$, using Theorem~\ref{azucena-maria-vila} we recover the formula for the
minimum distance of $C_\mathbb{X}(d)$ given in \cite[Theorem~5.5]{GR}. 
In these two cases formulas for the
basic parameters of $C_{\mathbb{X}_i}(d)$, $i=1,2$,  
are given in \cite[Theorem~1]{sorensen} and
\cite[Theorem~3.5]{ci-codes}, respectively. We also recover the
formulas for the second generalized
Hamming weight of some evaluation codes arising from complete
bipartite graphs given in \cite[Theorem~5.1]{G-SR-M} and
\cite[Theorem~3]{GRH} (see Corollary~\ref{jan9-15}).
\end{remark}

It turns out that the formula given in
Theorem~\ref{azucena-maria-vila}(e) 
is a far reaching generalization of the following result. 

\begin{corollary}{\cite[Theorem~5.1]{G-SR-M}}\label{jan9-15}
Let $\mathbb{X}$ be the Segre product
of two projective torus $\mathbb{X}_1$ and $\mathbb{X}_2$. 
Then the second generalized
Hamming weight of $C_{\mathbb{X}}(d)$ is given by 
$$
\delta_2(C_{\mathbb{X}}(d))=\min\{\delta_1(C_{\mathbb{X}_1}(d))\delta_2(C_{\mathbb{X}_2}(d)),
\delta_2(C_{\mathbb{X}_1}(d))\delta_1(C_{\mathbb{X}_2}(d))\}.
$$
\end{corollary}

\begin{remark} 
The knowledge of the regularity of $K[\mathbf{t}]/I(\mathbb{X})$ is important for applications to 
coding theory: for $d\geq {\rm reg}(K[\mathbf{t}]/I(\mathbb{X}))$ the
projective Segre code $C_\mathbb{X}(d)$ has minimum distance equal to
$1$ by Theorem~\ref{azucena-maria-vila}(f). Thus, potentially 
good projective Segre codes $C_\mathbb{X}(d)$ can occur only if $1\leq d < {\rm
reg}(K[\mathbf{t}]/I(\mathbb{X}))$. 
\end{remark}

\begin{definition} If $\mathbb{X}$ is parameterized by monomials
$z^{v_1},\ldots,z^{v_s}$, we say that $C_{\mathbb{X}}(d)$ is a
{\it parameterized projective code} of degree $d$.
\end{definition}

\begin{corollary} If $C_{\mathbb{X}_i}(d)$ is a parameterized
projective code of degree $d$ for $i=1,2$, then so is the corresponding projective Segre code
$C_{\mathbb{X}}(d)$.
\end{corollary}

\begin{proof} It suffices to observe that if $\mathbb{X}_1$ and 
$\mathbb{X}_2$ are parameterized by $z^{v_1},\ldots z^{v_s}$ and 
$w^{u_1},\ldots w^{u_r}$, respectively, then $\mathbb{X}$ is
parameterized by $z^{v_i}w^{u_j}$, $i=1,\ldots,s$, $j=1,\ldots,r$. 
\end{proof}

\noindent {\bf Acknowledgments.} We thank the referees for their
careful reading of the paper and for the improvements that
they suggested.

\bibliographystyle{plain}

\end{document}